\documentclass[11pt]{article}%
\usepackage{geometry}
\usepackage{amsmath,amsfonts,amsthm,amssymb, graphicx, setspace}
\usepackage{subfig}
\usepackage{color}
\usepackage{amsmath, amsfonts}
\usepackage{hyperref}
\usepackage{algpseudocode}
\usepackage{algorithm}
\usepackage{mathtools}
\usepackage{bbm}
\usepackage{amsmath}

\usepackage{amsfonts}
\usepackage{amssymb}
\usepackage{graphicx}%
\providecommand{\U}[1]{\protect\rule{.1in}{.1in}}
\DeclarePairedDelimiter\floor{\lfloor}{\rfloor}

\newtheorem{theorem}{Theorem}

\newtheorem{assumption}{Assumption}

\theoremstyle{example}
\newtheorem{lemma}{Lemma}

\newtheorem{proposition}{Proposition}
\newtheorem{remark}{Remark}
\theoremstyle{remark}
\newenvironment{proof}[1][Proof]{\noindent\textbf{#1.} }{\ \rule{0.5em}{0.5em}}
\begin{document}

\title{Exact Sampling of the Infinite Horizon Maximum of a Random Walk Over a
Non-linear Boundary}
\author{Blanchet, J.\thanks{Support from NSF grant DMS-132055 is gratefully acknowledged by J.
Blanchet.} \and Dong, J.  \and Liu, Z.}

\maketitle
\begin{abstract}
We present the first algorithm that samples 
$\max_{n\geq0}\{S_{n}-n^{\alpha}\},$
where $S_n$ is a mean zero random walk, 
and $n^{\alpha}$ with $\alpha\in(1/2,1)$ defines a nonlinear boundary. 
We show that our algorithm has finite expected running time. We also apply this
algorithm to construct the first exact simulation method for the steady-state departure process of a $GI/GI/\infty$
queue where the service time distribution has infinite mean.

\end{abstract}

\section{Introduction}

Consider a random walk $S_{n}=\sum_{i=1}^{n}X_{i}$ for $n\geq1$ and $S_{0}=0$,
where $\{X_{i}:i\geq1\}$ is a sequence of independent and identically
distributed random variables with $E[X_{1}]=0$. Without loss of generality, we
shall also assume that $Var(X_{1})=1$. Moreover, we shall impose the following
light-tail assumption on the distribution of $X_{i}$'s.

\begin{assumption}
\label{A1}There exists $\delta>0$, such that $E[\exp(\theta X_{1})]<\infty$
for $\forall \theta\in(-\delta,\delta)$.
\end{assumption}

In this paper, we develop the first algorithm that generates
perfect samples (i.e. samples without any bias) from the random variable
\[
M_{\alpha}=\max_{n\geq0}\{S_{n}-n^{\alpha}\},
\]
where $\alpha\in(1/2,1)$. Moreover, we will show that our algorithm has
finite expected running time. \newline

There has been substantial amount of work on exact sampling (i.e. sampling
without any bias) from the distribution of the maximum of a negative drifted
random walk, e.g. $M_{1}$ in our setting. Ensor and Glynn \cite{EnsGlyn:2000}
propose an algorithm to sample the maximum when the increments of the random
walk are light-tailed (i.e Assumption \ref{A1} holds). In \cite{BlanChen:2012}, Blanchet et al. propose an algorithm to simulate a multidimensional version of $M_{1}$ driven by Markov random walks. In \cite{BlanWall:2014}, Blanchet and Wallwater develop an algorithm to sample $M_1$ for the heavy-tailed case, which
requires only that $E\left\vert X_{1}\right\vert ^{2+\varepsilon}<\infty$ for
some $\varepsilon>0$ to guarantee finite expected termination time.

Some of this work is motivated by the fact that $M_{1}$ plays an important
role in ruin theory and queueing models. For example, the
steady state waiting time of $GI/GI/1$ queue has the same distribution as
$M_{1}$, where $X_{i}$ corresponds to the (centered) difference between
the $i$-th service time and the $i$-th interarrival time, (see \cite{Asm:2003}). Moreover, 
applying Coupling From The Past (CFTP), see for example
\cite{PropWil:1996} and \cite{Kendall:1998}, the techniques to sample
$M_{1}$ jointly with the random walk $\left\{  S_{n}:n\geq0\right\}  $ have
been used to obtain perfect sampling algorithms for more general queueing
systems, including multi-server queues \cite{BlanDonPei:2015}, infinite
server queues and loss networks \cite{BlanDong:2014}, and multidimensional
reflected Brownian motion with oblique reflection \cite{BlanChen:2012}.

The fact that $M_{\alpha}$ stochastically dominates $M_{1}$ makes the
development of a perfect sampler for $M_{\alpha}$ more difficult. For example,
the direct use of exponential tilting techniques as in
\cite{EnsGlyn:2000} is not applicable. However, similar to some of the
previous work, the algorithmic development uses the idea of record-breakers (see e.g. \cite{BlanDong:2014})
and randomization procedures
similar to the heavy-tailed context studied in \cite{BlanWall:2014}.

The techniques that we study here can be easily extended, using the techniques
studied in \cite{BlanChen:2012}, to obtain exact samplers of a
multidimensional analogue of $M_{\alpha}$ driven by Markov random walks (as
done in \cite{BlanChen:2012} for the case $\alpha=1$). Moreover, using the
domination technique introduced in Section 5 of \cite{BlanDonPei:2015}, the
algorithms that we present here can be applied to the case in which the
term $n^{\alpha}$ is replaced by $g\left(  n\right)  $ as long as there exists
$n_{0}>0$ such that $g\left(  n\right)  \geq n^{\alpha}$ for all $n\geq n_{0}$.

We mentioned earlier that algorithms which simulate $M_{1}$ jointly with
$\left\{  S_{n}:n\geq0\right\}  $ have been used in applications of CFTP.
Since the random variable $M_{\alpha}$ dominates $M_{1}$, and we also simulate
$M_{\alpha}$ jointly with $\{S_{n}:n\geq0\}$, we expect our results here to be
applicable to perfect sampling (using CFTP) for a wide range of processes. In
this paper, we will show how to use the ability to simulate $M_{\alpha}$
jointly with $\{S_{n}:n\geq0\}$ to obtain the first algorithm which samples
from the steady-state departure process of an infinite server queue in which
the job requirements have infinite mean; the case of finite mean service/job
requirements is treated in \cite{BlanDong:2014}.

The rest of the paper is organized as follows. In Section \ref{sec:main} we
discuss our sampling strategy. Then we provide a detailed running time
analysis in Section \ref{Sec_Running_Time}. As a sanity check, we implement
our algorithm in instances in which we can compute a theoretical benchmark,
this is reported in Section \ref{sec:num}. Finally, the application to exact
simulation of the steady-state departure process of an infinite server queue
with infinite mean service time is given in Section \ref{sec:queue}.

\section{Sampling strategy and main algorithmic development}

\label{sec:main} Our goal is to simulate $M_{\alpha}$ using a finite but
random number of $X_{i}$'s. To achieve this goal, we introduce the idea of
record-breakers.

Let $\psi(\theta):=\log E[\exp(\theta X_{i})]$. As $\psi(\theta)=\frac{1}%
{2}\theta^{2}+ o(\theta^{2})$, by Taylor expansion, there exists
$\delta^{\prime}<\delta$, such that $\psi(\theta)\leq\theta^{2}$, for
$\theta\in(-\delta^{\prime}, \delta^{\prime})$. Let
\begin{equation}
\label{eq:par}a< \min\left\{  4\delta^{\prime},\frac{1}{2}\right\}  ,
~~~\mbox{ and }~~~ b=\frac{4}{a}\log\left(  4\left(  \sum_{n=0}^{\infty}2^{n}
\exp(-a^{2} 2^{2n\alpha-n-4}\right)  \right).
\end{equation}
These choices of $a$ and $b$ will become clear in the proof of Lemma
\ref{lm:kappa}. We define a sequence of record-breaking times as $T_{0}:=0$.
For $k=1, 2\dots$, if $T_{k-1} <\infty$,
\[
T_{k}:=\inf\left\{  n>T_{k-1}: S_{n} > S_{T_{k-1}}+a (n-T_{k-1})^{\alpha} + b
(n-T_{k-1})^{1-\alpha}\right\}  ;
\]
otherwise if $T_{k-1}=\infty$, then $T_{k}=\infty$. We also define
\[
\kappa:=\inf\{k>0: T_{k}=\infty\}.
\]
Because the random walk has independent increments, $P(T_{i}=\infty|T_{i-1}%
<\infty)=P(T_{1}=\infty)$. Thus, $\kappa$ is a geometric random variable with
probability of success
\[
P(T_{1}=\infty).
\]
We first show that $\kappa$ is well defined.

\begin{lemma}
\label{lm:kappa} For $a$ and $b$ satisfying \eqref{eq:par},
\[
P(T_{1}=\infty) \geq\frac{3}{4}.
\]

\end{lemma}

\begin{proof}
We first notice that
\begin{align*}
P(T_{1}<\infty)  &  =\sum_{n=0}^{\infty}P(T\in\lbrack2^{n},2^{n+1}))\\
&  \le\sum_{n=0}^{\infty}\sum_{k\in\lbrack2^{n},2^{n+1})}P(S_{k}>ak^{\alpha
}+bk^{1-\alpha}).
\end{align*}
For any $k\in\lbrack2^{n},2^{n+1})$,
\begin{align*}
P(S_{k}>ak^{\alpha}+bk^{1-\alpha})  &  \leq\exp\left(  k\psi(\theta
)-\theta(ak^{\alpha}+bk^{1-\alpha})\right) \\
&  \leq\exp\left(  2^{n+1}\psi(\theta)-\theta a2^{\alpha n}-\theta
b2^{(1-\alpha)n}\right),
\end{align*}
for any $\theta\in(-\delta,\delta)$. If we set $\theta_{n}=a2^{(\alpha-1)n-2}%
$, as $a<4\delta^{\prime}$, $\theta_{n}<\delta^{\prime}$. Then
\begin{align*}
P\left(  S_{k}>ak^{\alpha}+bk^{1-\alpha}\right)   &  \leq\exp\left(
2^{n+1}\theta_{n}^{2}-\theta_{n}a2^{\alpha n}-\theta_{n}b2^{(1-\alpha
)n}\right) \\
&  =\exp\left(  -a^{2}2^{2n\alpha-n-3}-ab/4\right)  .
\end{align*}
Therefore,
\[
P(T_{1}<\infty)\leq\left(  \sum_{n=0}^{\infty}2^{n}\exp\left(  -a^{2}%
2^{2n\alpha-n-3}\right)  \right)  \exp\left(  -\frac{ab}{4}\right)  \leq
\frac{1}{4},
\]
where the last inequality follows from our choice of $b$.
\end{proof}

\bigskip

Let
\begin{equation}
\label{eq:xi}\xi:= \max_{n \ge 0}\left\{  (an^{\alpha}+bn^{1-\alpha})-\frac{1}%
{2}n^{\alpha} \right\}  .
\end{equation}
Conditional on the value of $\kappa$ and the values of $\{X_{i}: 1\leq i\leq
T_{\kappa-1}\}$, we define
\[
\Gamma(\kappa) := \left\lceil (2S_{T_{\kappa-1}}+2\xi)^{1/\alpha} \right\rceil
.
\]
The choice of $\xi$ will become clear in the proof of Lemma \ref{lm:rem}. We will next establish that
\[
M_{\alpha}=\max_{0\leq n\leq T_{\kappa-1}+\Gamma(\kappa)}\{S_{n}-n^{\alpha
}\}.
\]

\begin{lemma}
\label{lm:rem} For $n \geq T_{\kappa-1}+\Gamma(\kappa)$,
\[
S_{n} \leq n^{\alpha}.
\]

\end{lemma}

\begin{proof}
For $\xi$ defined in \eqref{eq:xi}, we have for any $n\geq0$,
\[
an^{\alpha}+bn^{1-\alpha}\leq\frac{1}{2}n^{\alpha}+\xi.
\]
Since $T_{\kappa}= \infty$, for $n \geq\Gamma(\kappa)$,
\begin{align*}
S_{T_{\kappa-1}+n}  &  \leq an^{\alpha}+bn^{1-\alpha}+S_{T_{\kappa-1}}\\
&  \leq\frac{1}{2}n^{\alpha}+\xi+\frac{1}{2} \Gamma(\kappa)^{\alpha} - \xi\\
&  \leq n^{\alpha}
\leq(T_{\kappa-1}+n)^{\alpha}.
\end{align*}

\end{proof}

\bigskip


Figure \ref{fig:alg} illustrates the basic idea of our algorithmic
development. The solid line is $n^{\alpha}$. The first dotted line from the
left (lowest dotted curve) is the record-breaking boundary that we start with,
$an^{\alpha-1}+bn^{\alpha}$. $T_{1}$ is the first record-breaking time. Based
on the value of $S_{T_{1}}$, we construct a new record-breaking boundary,
$S_{T_{1}}+a(n-T_{1})^{\alpha-1}+b(n-T_{1})^{\alpha}$ (the second dotted line
from the left). At time $T_{2}$, we have another record-breaker. Based on the
value of $S_{T_{2}}$, we construct again a new record-breaking boundary,
$S_{T_{2}}+a(n-T_{2})^{\alpha-1}+b(n-T_{2})^{\alpha}$ (the third dotted line
from the left). If from $T_{2}$ on, we will never break the record again
($T_{3}=\infty$), then we know that for $n$ large enough (say, $n>3000$ in the
figure), $S_{n}$ will never pass the solid boundary again.

\begin{figure}[tbh]
\caption{Bounds for record-breakers}%
\label{fig:alg}%
\centering
\includegraphics[width=0.5\textwidth]{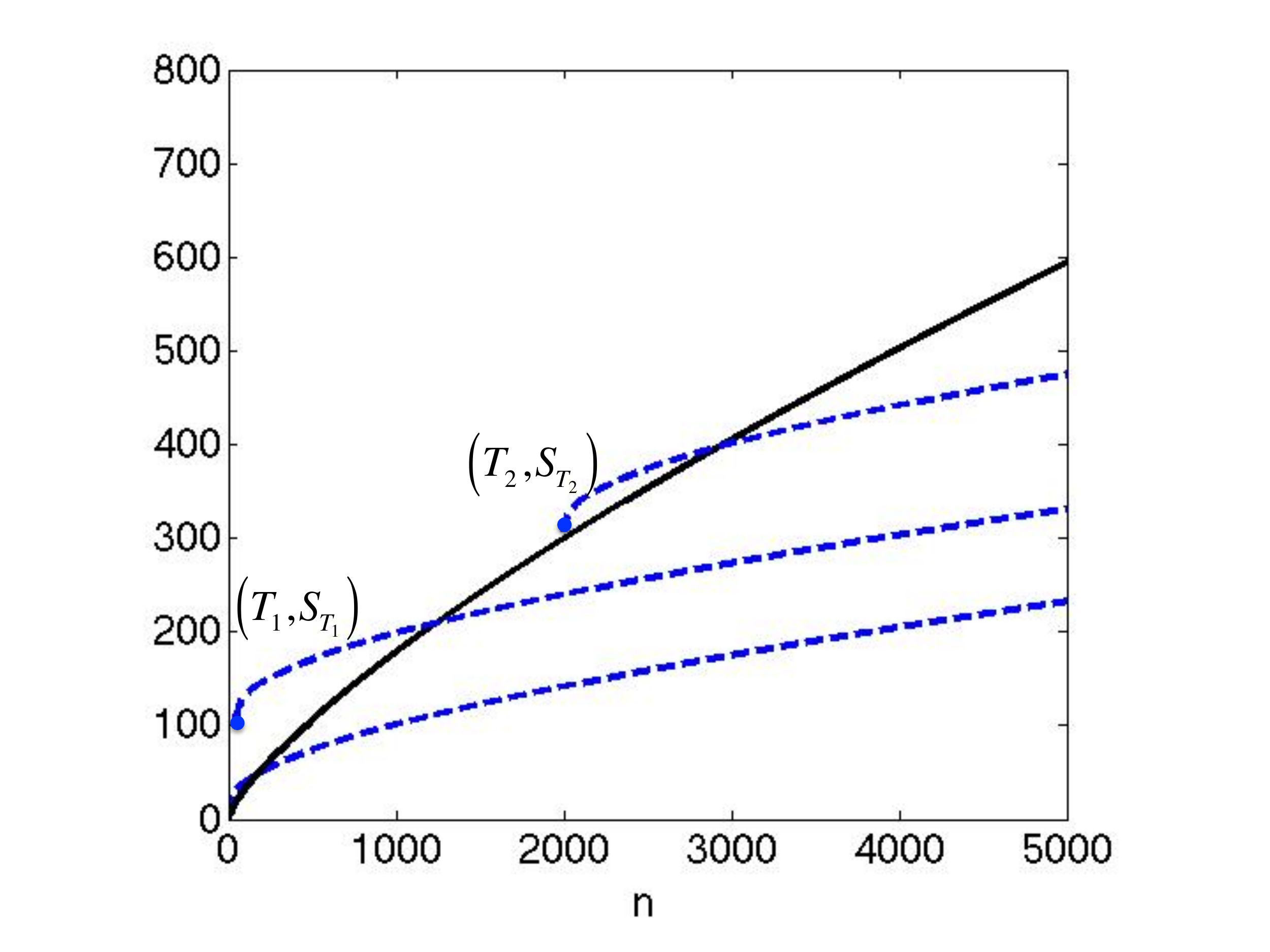}\end{figure}

The actual simulation strategy goes as follows.\newline

\noindent\textbf{Algorithm 1.} Sampling $\Gamma(\kappa)$ together with
$(X_{i}: 1\leq i\leq T_{\kappa-1}+\Gamma(\kappa))$.

\begin{enumerate}
\item[i)] Initialize $T_{0}=0$, $k=1$.

\item[ii)] For $T_{k-1}<\infty$, sample $J \sim$Bernoulli$(P(T_{k}=\infty|T_{k-1})).$

\item[iii)] If $J=0$, sample $(X_{i}: i=T_{k-1}+1, \dots, T_{k})$ conditional
on $T_{k}<\infty$. Set $k=k+1$ and go back to step ii); otherwise ($J=1$), set
$\kappa=k$ and go to step iv).

\item[iv)] Calculate $\Gamma(\kappa)$, sample $(X_{i}: i=T_{k-1}+1, \dots,
T_{k-1} + \Gamma(\kappa))$ conditional on $T_{k}=\infty$.
\end{enumerate}

\bigskip

\begin{remark}
\label{rem_ab}
In general, any $a<\min\left\{  4\delta^{\prime},1/2\right\}  $, and
$b\geq\frac{4}{a}\log\left(  4\left(  \sum_{n=0}^{\infty}2^{n}\exp(-a^{2}%
2^{2n\alpha-n-4}\right)  \right)  $ would work. However, there is a trade-off.
The larger the value of $a$ and $b$, the smaller the value of $\kappa$,
but the value of $\Gamma(\kappa)$ would be larger. We conduct a numerical
study on the choice of these parameters in Section \ref{sec:num1}.
\end{remark}

In what follows, we shall elaborate on how to carry out step ii), iii) and iv)
in Algorithm 1. In particular, step ii) and iii) are outlined in Procedure A.
Step iv) is outlined in Procedure B.

\subsection{Step ii) and iii) in Algorithm 1}

\label{sec:step1} It turns out step ii) and iii) can be carried out
simultaneously using exponential tilting based on the results and proof of Lemma
\ref{lm:kappa}.

We start by explaining how to sample the first record-breaking time $T_{1}$. We introduce an
auxiliary random variable $N$ with probability mass function (pmf)
\begin{equation}
\label{eq1}p(n)=P\left(  N=n\right)  =\frac{2^{n} \exp\left(  -a^{2}
2^{2n\alpha-n -3} \right)  }{\sum_{m=0}^{\infty}2^{m} \exp\left(  -a^{2}
2^{2m\alpha-m -3} \right)  }, ~~~\mbox{ for $n\geq1$}
\end{equation}
We can then apply exponential tilting to sample the path $(X_{1},X_{2} \dots, X_{T_{1}%
})$ conditional on $T_{1}<\infty$. The actual sampling algorithm goes as
follows.\newline

\noindent\textbf{Procedure A.} Sampling $J$ with $P(J=1)=P(T_{1}=\infty)$; if
$J=0$, output $(X_{1}, \dots, X_{T_{1}})$.

\begin{enumerate}
\item[i)] Sample a random time $N$ with pmf \eqref{eq1}.

\item[ii)] Let $\theta_{N}=a2^{N(\alpha-1)-2}$. Generate $X_{1}, X_{2}, \dots,
X_{2^{N+1}-1}$ under exponential tilting with tilting parameter $\theta_{N}$,
i.e.
\[
\frac{dP_{\theta_{N}}}{dP}1\{X_{i} \in A\}=\exp(\theta_{N} X_{i}-\psi
(\theta_{N}))1\{X_{i}\in A\}.
\]
Let $T_{1}=\inf\{n\geq1: S_{n}>an^{\alpha}+bn^{1-\alpha}\} \wedge2^{N+1}$

\item[iii)] Sample $U\sim$Uniform$[0,1]$. If
\[
U\leq\frac{\exp(-\theta_{N} S_{T_{1}} + T_{1}\psi(\theta_{N}))}{p(N)}I\left\{
T_{1}\in\left[  2^{N}, 2^{N+1}\right)  \right\}  ,
\]
then set $J=0$ and output $(X_{1}, X_{2}, \dots, X_{T_{1}})$; else, set
$J=1$.\newline
\end{enumerate}

We next show that Procedure A works.

\begin{theorem}
\label{th:step1} In Procedure A, $J$ is a Bernoulli random variable with
probability of success $P(T_{1}=\infty)$. If $J=0$, the output $(X_{1}, X_{2},
\dots, X_{T_{1}})$ follows the distribution of $(X_{1}, X_{2}, \dots,
X_{T_{1}})$ conditional on $T_{1}<\infty$.
\end{theorem}

\begin{proof}
We first show that the likelihood ratio in step iii) is well-defined.
\begin{align*}
&  \frac{\exp(-\theta_{n} S_{T_{1}} + T_{1} \psi(\theta_{n}))}{P\left(
N=n\right)  } I\{T_{1}\in[2^{n}, 2^{n+1} )\}\\
\le &  \frac{\exp(-\theta_{n} (a 2^{\alpha n}+b 2^{(1-\alpha)n})+ 2^{n+1}
\theta_{n}^{2}))}{P(N=n)}\\
=  &  \frac{\exp\left(  -a^{2} 2^{2n\alpha-n -3}-ab/4 \right)  }{P(N=n)}\\
=  &  2^{-n} \exp(-ab/4) \sum_{m=0}^{\infty}2^{m} \exp(-a^{2} 2^{2m\alpha-m
-3})
\le \frac{1}{4},
\end{align*}
where the last inequality follows from our choice of $b$ as in the proof of
Lemma \ref{lm:kappa}. We next prove that $P(J=0)=P(T_{1}<\infty)$.
\begin{align*}
E[I\{J=0\}|N=n]  &  =E_{\theta_{n}}\left[  I\left\{  U\leq\frac{\exp
(-\theta_{n} S_{T_{1}} + T_{1}\psi(\theta_{n}))}{p(n)}\right\}  I\{T_{1}%
\in[2^{n}, 2^{n-1})\}\right] \\
&  = E_{\theta_{n}}\left[  \frac{\exp(-\theta_{n} S_{T_{1}} + T_{1}\psi
(\theta_{n}))}{p(n)}I\{T_{1}\in[2^{n}, 2^{n+1})\}\right] \\
&  =\frac{P(T_{1}\in[2^{n}, 2^{n+1}))}{p(n)}.%
\end{align*}
Thus,
\begin{align*}
E[I\{J=0\}]  &  =\sum_{n=0}^{\infty}E[I\{J=0\}|N=n]p(n)\\
&  = \sum_{n=0}^{\infty}P(T_{1}\in[2^{n}, 2^{n+1}))=P(T_{1}<\infty).
\end{align*}
Let $P^{*}(\cdot)$ denote the measure induced by Procedure A. We next show
that $P^{*}(X_{1} \in A_{1}, \ldots, X_{t} \in A_{t} | J=0 )=P(X_{1}\in A_{1},
\ldots, X_{t} \in A_{t} |T<\infty),$ where $t$ is a positive integer, and
$A_{i}\subset\mathbb{R}$, $i=1,2,\dots, t$, is a sequence of Borel measurable
sets satisfying $A_{i} \subset\{x\in \mathbb R: x\leq ai^{\alpha}+bi^{1-\alpha}\}$ for
$i<t$ and $A_{t} \subset\{x \in \mathbb R: x>at^{\alpha}+bt^{1-\alpha}\}$. Let
$n_{t}:=\floor{\log_2 t}$.
\begin{align*}
&  P^{*}(X_{1} \in A_{1}, \ldots, X_{t} \in A_{t} | J=0 )\\
=  &  \frac{P^{*}(X_{1} \in A_{1}, \ldots, X_{t} \in A_{t} , J=0 )}{P(J=0)}\\
=  &  \frac{P(N=n_{t})}{P(T_{1}<\infty)}E_{\theta_{n_{t}}} \left[  I \left\{
X_{1} \in A_{1}, \ldots, X_{t} \in A_{t} \right\}  I \left\{  U \leq\frac
{\exp(-\theta_{n_{t}} S_{t}+t \psi(\theta_{n_{t}}))}{p(n_{t})} \right\}
\right] \\
=  &  \frac{p(n_{t})}{P(T_{1}<\infty)} E_{\theta_{n_{t}}}\left[  I \left\{
X_{1} \in A_{1}, \ldots, X_{t} \in A_{t} \right\}  \frac{\exp(-\theta_{n_{t}}
S_{t}+t \psi(\theta_{n_{t}}))}{p(n_{t})} \right] \\
=  &  \frac{E\left[  I \left\{  X_{1} \in A_{1}, \ldots, X_{t} \in A_{t}
\right\}  \right]  }{P(T_{1}<\infty)}\\
=  &  P(X_{1}\in A_{1}, \ldots, X_{t} \in A_{t} |T_{1}<\infty).
\end{align*}

\end{proof}

The extension from $T_{1}$ to $T_{k}$ is straightforward: because for
$T_{k-1}<\infty$, given the value of $T_{k-1}$ and $S_{T_{k-1}}$, we
essentially start the random walk afresh from each $T_{k-1}$ on. Thus, to
execute step ii) and iii) in Algorithm 1, given $T_{k-1}<\infty$, we can apply
Procedure A. Based on the output, if $J=0$, we denote $(\tilde{X}_{1},\tilde{X}_{2},\dots,\tilde
{X}_{T})$ as the output from Procedure A, and set $(X_{T_{k-1}+1}%
,\dots,X_{T_{k-1}+T})=(\tilde{X}_{1},\dots,\tilde{X}_{T})$ and $T_{k}%
=T_{k-1}+T$, otherwise, set $\kappa=k$.

\subsection{Step iv) in Algorithm 1}

Sampling $(X_{1}, \dots, X_{T_{\kappa-1}})$ is realized by iteratively
applying Procedure A until it outputs $J=1$. Once we found $\kappa$, sampling
$(X_{T_{\kappa-1}+1}, \dots, X_{T_{\kappa-1}+n})$ requires us to sample the
trajectory of the random walk conditioning on that it never passes the
non-linear upper bound. In particular, given $\kappa=k$, we would like to
sample $(X_{T_{\kappa-1}+1}, \dots, X_{T_{\kappa-1}+n})$ from $P(\cdot|
\mathcal{F}_{k-1}, T_{k}=\infty)$, where $\{\mathcal{F}_{k}: k \geq0\}$ denote
the filtration generated by the random walk. We can achieve this conditional
sampling using the acceptance-rejection technique.

We first introduce a method to simulate a Bernoulli random variable with
probability of success $P(T_{1}=\infty|T_{1}>t,S_{t})$, which follows a
similar exponential tilting idea as that used in Section \ref{sec:step1}.

Let
\[
\tilde T_{t, s}:=\inf\left\{  n \geq0: s + S_{n} > a(n+t)^{\alpha
}+b(n+t)^{1-\alpha}\right\}  .
\]
Given $t$, we introduce an auxiliary random variable $\tilde N(t)$ with pmf
\begin{equation}
\label{eq2}p_{t}(n)=P\left(  \tilde N(t)=n\right)  =\frac{2^{n} \exp\left(
-2^{-n-4}a^{2} (2^{n}+t)^{2\alpha} \right)  }{\sum_{m=0}^{\infty}2^{m}
\exp\left(  -2^{-m-4}a^{2} (2^{m}+t)^{2\alpha} \right)  },
~~~\mbox{ for $n\geq1$}.
\end{equation}
Given $\tilde N(t)=n$, we apply exponential tilting to sample $(\tilde X_{1},
\tilde X_{2}, \dots, \tilde X_{2^{n+1}-1})$, with tilting parameter
\[
\tilde\theta_{n}(t) =2^{-n-2}a (2^{n}+t)^{\alpha},
\]
i.e.
\[
\frac{dP_{\tilde\theta_{n}(t)}}{dP}1\{X_{i} \in A\}=\exp(\tilde\theta_{n}(t)
X_{i}-\psi(\tilde\theta_{n}(t)))1\{X_{i}\in A\}.
\]
We also define $\tilde S_{k}:= \tilde X_{1}+\dots+\tilde X_{k}$ for $k\geq1$, and $\tilde T=\inf\left\{  n \geq0: s + \tilde S_{n} > a(n+t)^{\alpha
}+b(n+t)^{1-\alpha}\right\}\wedge 2^{n+1}$. In what follows, we shall suppress the
dependence on $t$ when there is no confusion. Let
\begin{equation}
\label{eq:ber}\tilde J=1-I\left\{  U\leq\frac{\exp(-\tilde\theta_{n} \tilde{S}_{\tilde
T} + \tilde T\psi(\tilde\theta_{n}))}{p_{t}(n)}I\left\{  \tilde T\in\left[
2^{n}, 2^{n+1}\right)  \right\}  \right\},
\end{equation}
where $U\sim$Uniform$[0,1]$.

\begin{lemma}
\label{lm:Ber2} For $\tilde J$ defined in \eqref{eq:ber}, when $s<\frac
{a}{4} t^{\alpha}$, we have
\[
P\left(  \tilde J=1\right)  =P\left(  \tilde T_{t,s}=\infty\right)  .
\]

\end{lemma}

\begin{proof}
We first notice that
\begin{align*}
&  \frac{\exp(-\tilde\theta_{n} S_{\tilde T} + \tilde T\psi(\tilde\theta
_{n}))}{p_{t}(n)}I\left\{  \tilde T\in\left[  2^{n}, 2^{n+1}\right)  \right\}
\\
&  \leq\frac{1}{p_{t}(n)}\exp\left(  - \tilde\theta_{n} \left(  a(2^{n}%
+t)^{\alpha}+b(2^{n}+t)^{1-\alpha}-\frac{a}{4}t^{\alpha}\right)
+2^{n+1}\tilde\theta_{n}^{2} \right) \\
&  \leq\frac{1}{p_{t}(n)}\exp\left(  -2^{-n-3}a^{2} (2^{n}+t)^{2\alpha
}+2^{-n-4}a^{2}(2^{n}+t)^{2\alpha}-\frac{ab}{4}\right) \\
&  = \frac{1}{p_{t}(n)}\exp\left(  -2^{-n-4}a^{2} (2^{n}+t)^{2\alpha}%
-\frac{ab}{4}\right) \\
&  \leq\left(  \sum_{m=0}^{\infty}2^{m} \exp\left(  -2^{-m-4}a^{2}
(2^{m}+t)^{2\alpha} \right)  \right)  \times\exp(-ab/4)\\
&  \leq\left(  \sum_{m=0}^{\infty}2^{m} \exp\left(  -a^{2} 2^{2m\alpha-m-4}
\right)  \right)  \times\exp(-ab/4)  
\leq\frac{1}{4},
\end{align*}
where the last inequality follows from our choice of $a$ and $b$. The rest of
the proof follows exact the same steps as the proof of Theorem \ref{th:step1}.
We shall omit it here.
\end{proof}

Let
\[
L(n)=\inf\left\{  k \geq n: S_{k} > ak^{\alpha}+bk^{1-\alpha} \mbox{ or }
S_{k}<\frac{a}{4}k^{\alpha}\right\}  .
\]
The sampling algorithm goes as follows.\newline

\noindent\textbf{Procedure B.} Sampling $(X_{1}, \dots, X_{n})$ conditional on
$T_{1}=\infty$.

\begin{enumerate}
\item[i)] Sample $(X_{1},\dots, X_{n})$ under the nominal distribution
$P(\cdot)$.

\item[ii)] If $\max_{1\leq k\leq n}\{S_{k}-ak^{\alpha}-bk^{1-\alpha}\}>0$, go
back to step i); else, go to step iii).

\item[iii)] Sample $L(n)$ and $(X_{n+1}, X_{n+2}, \dots, X_{L(n)})$ under the
nominal distribution $P(\cdot)$. If $S_{L(n)}> aL(n)^{\alpha}+bL(n)^{1-\alpha
}$, go back to step i); else, go to step iv).

\item[iv)] Sample $\tilde N$ with probability mass function $p_{L(n)}$ defined
in \eqref{eq2}. Generate $(\tilde X_{1}, \tilde X_{2}, \dots, \tilde
X_{2^{\tilde N+1}-1})$ under exponential tilting with tilting parameter
$\tilde\theta_{\tilde N}=2^{\tilde N-2} a\left(  2^{\tilde N}+L(n)\right)
^{\alpha}$. Let $\tilde T=\inf\{k\geq1: S_{L(n)}+\tilde S_{k}%
>a(k+L(n))^{\alpha}+b(k+L(n))^{1-\alpha}\} \wedge2^{\tilde N+1}$.

\item[v)] Sample $U\sim$Uniform$[0,1]$. If
\[
U\leq\frac{\exp\left(  -\tilde\theta_{\tilde N} S_{\tilde T} + \tilde
T\psi(\tilde\theta_{\tilde N})\right)  }{p_{t}(\tilde N)}I\left\{  \tilde
T\in\left[  2^{\tilde N}, 2^{\tilde N+1}\right)  \right\}  ,
\]
set $\tilde J=0$ and go back to Step i); else, set $\tilde J=1$ and output
$(X_{1}, \dots, X_{n})$.\newline
\end{enumerate}

We next show that Procedure B works.

\begin{theorem}
\label{th:step2} The output of Procedure B follows the distribution of
$(X_{1}, \dots, X_{n})$ conditional on $T_{1}=\infty$.
\end{theorem}

\begin{proof}
Let $P^{\prime}(\cdot)=P(\cdot|T_{1}=\infty)$. We first notice that
\[
\frac{dP^{\prime}}{dP}(X_{1}, \dots, X_{n})=\frac{I\{T_{1}>n\}P(T_{1}%
=\infty|S_{n}, T_{1}>n)}{P(T_{1}=\infty)}\leq\frac{1}{P(T_{1}=\infty)}.
\]
Let $P^{\prime\prime}(\cdot)$ denote the measure induced by Procedure B. Then
we have, for any sequence of Borel measurable sets $A_{i}\subset\mathbb{R}$,
$i=1,2, \dots, n$,
\begin{align*}
&  P^{\prime\prime}\left(  X_{1}\in A_{1}, \dots, X_{n}\in A_{n}\right) \\
&  =P\left(  X_{1}\in A_{1}, \dots, X_{n}\in A_{n}| T_{1}>L(n), \tilde
J=1\right) \\
&  =P\left(  X_{1}\in A_{1}, \dots, X_{n}\in A_{n}| T_{1}>L(n), \tilde
T_{L(n), S_{L(n)}}=\infty\right) \\
&  =P(X_{1}\in A_{1}, \dots, X_{n}\in A_{n}|T_{1}=\infty),
\end{align*}
where the second equality follows from Lemma \ref{lm:Ber2}, and the third
equality follows from the fact that
\[
P(T_{1}=\infty|S_{t}, T_{1}>t)=P\left(  \tilde T_{t, S_{t}}=\infty\right)  .
\]

\end{proof}

To execute Step iv) in Algorithm 1, we apply Procedure B with $n=\Gamma
(\kappa)$.

\section{Running time analysis\label{Sec_Running_Time}}

In this section, we provide a detailed running time analysis of Algorithm 1.

\begin{theorem}
Algorithm 1 has finite expected running time.
\end{theorem}

We divide the analysis into the following steps.

\begin{enumerate}
\item From Lemma \ref{lm:kappa}, the number of iterations between step ii) and
iii) follows a geometric distribution with probability of success
$P(T_{1}=\infty)$.

\item In each iteration (when applying Procedure A), we will show that the
length of the path needed to sample $J$ has finite moments of all orders
(Lemma \ref{lm:run_A}).

\item For step iv), we will show that $\Gamma(\kappa)$ has finite moments of
all orders (Lemma \ref{lm:run4}).

\item When applying Procedure B for step iv), we will first show that
$L(\Gamma(\kappa))$ has finite moments of all orders (Lemma \ref{lm:run4}).

\item When applying Procedure B for step iv), we will also show that the
length of the path needed to sample $J$ has finite moments of all orders
(Lemma \ref{lm:run_B}).
\end{enumerate}

\begin{lemma}
\label{lm:run_A} The length of the path needed to sample the Bernoulli $J$ in
Procedure A has finite moments of every order.
\end{lemma}

\begin{proof}
The length of the path generated in Procedure A is bounded by $2^{N+1}$, where
$N$'s distribution is defined in \eqref{eq1}. Therefore, $\forall r > 0$,
\begin{align*}
E[2^{(N+1)r}]  &  =\frac{\sum_{m=0}^{\infty}2^{(m+1)r}2^{m} \exp\left(  -a^{2}
2^{2m\alpha-m -3} \right)  }{\sum_{m=0}^{\infty}2^{m} \exp\left(  -a^{2}
2^{2m\alpha-m -3} \right)  }\\
&  = \frac{\sum_{m=0}^{\infty}\exp\left(  -a^{2} 2^{2m\alpha-m -3} +
(mr+r+m)\log2 \right)  }{\sum_{m=0}^{\infty}2^{m} \exp\left(  -a^{2}
2^{2m\alpha-m -3} \right)  }.\\
\end{align*}
Since for $m$ sufficiently large,
\[
\exp\left(  -a^{2} 2^{2m\alpha-m -3} +(mr+r+m)\log2 \right)  \le\exp\left(
-a^{2} 2^{2m\alpha-m -4} \right)  ,
\]
for fixed $r>0$, $\exists C>0$, such that
\begin{align*}
&  \frac{\sum_{m=0}^{\infty}\exp\left(  -a^{2} 2^{2m\alpha-m -3} +
(mr+r+m)\log2 \right)  }{\sum_{m=0}^{\infty}2^{m} \exp\left(  -a^{2}
2^{2m\alpha-m -3} \right)  }\\
\le &  C\frac{\sum_{m=0}^{\infty}\exp\left(  -a^{2} 2^{2m\alpha-m -4} \right)
}{\sum_{m=0}^{\infty}2^{m} \exp\left(  -a^{2} 2^{2m\alpha-m -3} \right)  } <
\infty.
\end{align*}
Note that this also implies that
\[
E[T^{r} I(T < \infty)] \le E[2^{(N+1)r}I (J = 0)] \le E[2^{(N+1)r}] < \infty.
\]

\end{proof}

\begin{lemma}
\label{lm:run4} $\Gamma(\kappa)$ and $L(\Gamma(\kappa))$ have finite moments
of any order.

\end{lemma}

\begin{proof}
We start with $\Gamma(\kappa)$. Let $R_{n}:=S_{n}-an^{\alpha}-bn^{1-\alpha}$.
For $T_{i}<\infty$, we also denote
\[
\mathcal{R}_{i}:=S_{T_{i}}-S_{T_{i-1}}-a(T_{i}-T_{i-1})^{\alpha}-b
(T_{i}-T_{i-1})^{1-\alpha}.
\]
\begin{align*}
\Gamma(\kappa)  &  =\left\lceil (2S_{T_{\kappa-1}} + 2\xi)^{1/\alpha}
\right\rceil \\
&  \leq\left\lceil \left(  \sum_{i=1}^{\kappa-1}(T_{i}-T_{i-1})^{\alpha}
+2\kappa\xi+ 2\sum_{i=1}^{\kappa-1}\mathcal{R}_{i}\right)  ^{1/\alpha
}\right\rceil.
\end{align*}
We first prove that conditioning on $T_{1}<\infty$, $R_{T_{1}}$ has finite
moments of every order.
\begin{align*}
&  E[e^{\gamma R_{T_{1}}} I (T_{1} < \infty)]\\
=  &  \sum_{n = 0}^{\infty}E[e^{\gamma R_{n}}I(T_{1}=n)]\\
=  &  \sum_{n = 0}^{\infty}E[e^{\gamma(X_{n}+S_{n-1}-an^{\alpha}-bn^{1-\alpha
})}I(T_{1}=n)]\\
\le &  \sum_{n = 0}^{\infty}E[e^{\gamma X_{n}}I(T_{1}=n)]\\
\le &  \sum_{n = 0}^{\infty}E[e^{p \gamma X_{n}}]^{1/p} E[I(T_{1}=n)]^{1/q}
~~\mbox{ for } p, q >1, \frac{1}{p}+\frac{1}{q}=1
\mbox{ by H\"older's inequality}\\
\le &  E[e^{p \gamma X_{n}}]^{1/p} \sum_{n = 0}^{\infty}P(T_{1}=n)^{1/q} .
\end{align*}
Because $X_{n}$ has moment generating function within a neighborhood of $0$,
we can choose $p>0$ and $\gamma>0$ such that $E[e^{p \gamma X_{n}}%
]^{1/p}<\infty$. In the proof of Lemma \ref{lm:run_A} we showed that $\forall
r >0$, $E[T_{1}^{r} I(T_{1}<\infty)]<\infty$, which implies that
$P(T_{1}=n)=O(\frac{1}{n^{r}})$. Because $r$ can be any positive value,
$\sum_{n = 0}^{\infty}P(T_{1}=n)^{1/q} < \infty$.
By Jensen's inequality, for any $r \ge1$,
\begin{align}
\label{eq:bound} &  E\left[  \left(  \sum_{i=1}^{\kappa-1}(T_{i}%
-T_{i-1})^{\alpha}+2\kappa\xi+2\sum_{i=1}^{\kappa-1}\mathcal{R} _{i}\right)
^{r/\alpha} \right] \nonumber\\
&  \le3^{\frac{r}{\alpha}-1}E\left[  \left(  \sum_{i=1}^{\kappa-1}%
(T_{i}-T_{i-1})^{\alpha}\right)  ^{r/\alpha} + \left(  2 \kappa\xi\right)
^{r/\alpha}+ \left(  2\sum_{i=1}^{\kappa-1}\mathcal{R} _{i} \right)
^{r/\alpha} \right].
\end{align}
We next analyze each of the three part on the right hand side of
\eqref{eq:bound}. As $\kappa$ is a geometric random variable, $E[\left(  2
\kappa\xi\right)  ^{r/\alpha}] < \infty$.
\begin{align*}
E\left[  \left(  \sum_{i=1}^{\kappa-1}(T_{i}-T_{i-1})^{\alpha}\right)
^{r/\alpha} \right]   &  = E \left[  E\left[  \left(  \sum_{i=1}^{\kappa
-1}(T_{i}-T_{i-1}) ^{\alpha}\right)  ^{r/\alpha}\lvert\kappa\right]  \right]
\\
&  \le E \left[  \kappa^{r/\alpha-1} E\left[  \sum_{i=1}^{\kappa-1}%
(T_{i}-T_{i-1})^{r}\lvert\kappa\right]  \right] \\
&  \le E \left[  \kappa^{r/\alpha} 2^{(N+1)r+1} \right] \\
&  =E \left[  \kappa^{r/\alpha}\right]  E\left[  2^{(N+1)r+1} \right]  <
\infty.
\end{align*}
Similarly, we have
\[
E\left[  \left(  2\sum_{i=1}^{\kappa-1}\mathcal{R} _{i} \right)  ^{r/\alpha}
\right]  \le E \left[  (2\kappa)^{r/\alpha} \right]  E\left[  R_{T_{1}%
}^{r/\alpha}|T_{1}<\infty\right]  \newline< \infty.
\]
Therefore, we have
\[
E[\Gamma(\kappa)]<\infty.
\]
We next analyze $L(\Gamma(\kappa))$.
\[
L(\Gamma(\kappa))-\Gamma(\kappa) \le\inf\left\{  n \ge0: S_{n+\Gamma(\kappa)}
<\frac{a}{4} (n+\Gamma(\kappa))^{\alpha}\right\}  .
\]
Given $\Gamma(\kappa)=n_{*}$ and $S_{\Gamma(\kappa)}=s_{*}$, since $s_{*}<a
n_{*}^{\alpha}+b n_{*}^{1-\alpha}$,
\begin{align*}
&  P(L(\Gamma(\kappa))-\Gamma(\kappa) > n|\Gamma(\kappa)=n_{*}, S_{\Gamma
(\kappa)}=s_{*})\\
\le &  P\left(  {S}_{n} \ge\frac{a}{4}(n+n_{\ast})^{\alpha}-s_{*}\right) \\
\le &  P\left(  {S}_{n} \ge\frac{a}{4}(n+n_{\ast})^{\alpha}-a n_{\ast}%
^{\alpha}-b n_{\ast}^{1-\alpha}\right) \\
\le &  P\left(  {S}_{n} \ge\frac{a}{4}(n+n_{\ast})^{\alpha}-\frac{1}{2}
n_{\ast}^{\alpha}-\xi\right) \\
\le &  \exp\left(  n \theta^{2}-\theta\left(  \frac{a}{4}(n+n_{\ast})^{\alpha
}-\frac{1}{2} n_{\ast}^{\alpha}-\xi\right)  \right)
\mbox{ for $0<\theta<\delta^{\prime}$}.
\end{align*}
Let $w_{n}=\frac{a}{4}(n+n_{\ast})^{\alpha}-\frac{1}{2} n_{\ast}^{\alpha}-\xi
$. If we pick $\theta=\epsilon_{n} |\frac{w_{n}}{n}|$ where $\epsilon_{n}$ is
chosen such that $\theta<\delta^{\prime}$. Then
\[
\exp\left(  n \theta^{2}-\theta w_{n}\right)  \leq\exp\left(  -\frac{w_{n}%
^{2}}{n}\epsilon_{n}(1-\epsilon_{n})\right)  \leq\exp\left(  -\frac{w_{n}^{2}%
}{4n}\right)  .
\]
We notice that for $n$ large enough,
\[
w_{n}\leq\frac{a}{5}(n+n_{*})^{\alpha}.
\]
Thus, $\exists C>0$, such that
\begin{align*}
P(L(\Gamma(\kappa))-\Gamma(\kappa) > n|\Gamma(\kappa)=n_{*}, S_{\Gamma
(\kappa)}=s_{*})  &  \leq C\exp\left(  -\frac{a^{2}}{100}\frac{(n+n_{*}%
)^{2\alpha}}{n}\right) \\
&  \leq C\exp\left(  -\frac{a^{2}}{100} n^{2\alpha-1}\right).
\end{align*}
This implies that, given $\Gamma(\kappa)$ and $S_{\Gamma(\kappa)}$,
$L(\Gamma(\kappa))-\Gamma(\kappa)$ has finite moments of all orders.

\end{proof}

\begin{lemma}
\label{lm:run_B} The length of the path needed to sample the Bernoulli $\tilde
J$ in Procedure B has finite moments of every order.

\end{lemma}

\begin{proof}
The length of the path in Procedure B is bounded by $2^{\tilde N(t)+1}$, with
$\tilde N(t)$ sampled from \eqref{eq2}. For any $r>0$,
\begin{align*}
E[2^{(\tilde N(t)+1)r}]  &  =\frac{\sum_{m=0}^{\infty}2^{(m+1)r}2^{m}
\exp\left(  -2^{-m-4}a^{2} (2^{m}+t)^{2\alpha} \right)  }{\sum_{m=0}^{\infty
}2^{m} \exp\left(  -2^{-m-4}a^{2} (2^{m}+t)^{2\alpha} \right)  }\\
&  = \frac{\sum_{m=0}^{\infty}\exp\left(  -2^{-n-4}a^{2} (2^{m}+t)^{2\alpha}
+(mr+m+r)\log2\right)  }{\sum_{m=0}^{\infty}2^{m} \exp\left(  -2^{-n-4}a^{2}
(2^{m}+t)^{2\alpha} \right)  }.
\end{align*}
As for $m$ sufficiently large,
\[
\exp\left(  -2^{-m-4}a^{2} (2^{m}+t)^{2\alpha} +(mr+m+r)\log2\right)  <
\exp\left(  -2^{-m-5}a^{2} (2^{m}+t)^{2\alpha} \right)  ,
\]
for fixed $r>0$, $\exists C>0$, such that
\[
E[2^{(\tilde N(t)+1)r}] \le C \frac{\sum_{m=0}^{\infty}\exp\left(
-2^{-m-5}a^{2} (2^{m}+t)^{2\alpha} \right)  }{\sum_{m=0}^{\infty}2^{m}
\exp\left(  -2^{-m-4}a^{2} (2^{m}+t)^{2\alpha} \right)  } < \infty.
\]

\end{proof}

\subsection{Numerical experiments}

\label{sec:num} In this section, we conduct numerical experiments to analyze
the performance of Algorithm 1 for different values of parameters. We will
also conduct a sanity check of the correctness of our algorithms (empirically)
by simulating the steady state departure process of an infinite server
queueing model. The details of the simulation of the infinite server queue
will be given in Section \ref{sec:queue}.

\subsubsection{Choice of parameter $a$}

\label{sec:num1}  In Remark~\ref{rem_ab} we briefly discuss how the parameters $a$ and $b$ would affect
the performance of Algorithm 1. We shall fix the value of $b$ upon
our choice of $a$ as in \eqref{eq:par}, as we want to guarantee that probability of record-breaking is small enough, while keeping $\Gamma(\kappa)$ as small as possible.
In this subsection, we conduct some simulation experiments to analyze the effect of different values of $a$ on the computational cost.  
We first notice that the choice of $a$ and $b$ will affect the
distribution of $N$, which is the length of trajectory generated in Procedure
A. In Procedure B, the value of $\Gamma(\kappa)$, $L(\Gamma(\kappa))$ and the
distribution of $\tilde{N}$ also depend on the value of $a$ and $b$. 

Let $X_{i} \,{\buildrel d \over =}\, X-1$, where
$X$ is a unit rate exponential random variable. 
Then $\psi(\theta)=-\theta-\log(1-\theta)$, for $\theta<1$.
Let $g(\theta):=\psi(\theta)-\theta^{2}$. As $g^{\prime}(0)=0$,
$g^{\prime\prime}(\theta)=\frac{1}{(1-\theta)^{2}}-2$, we have
\[
g(\theta)<0 \;\;\;\; \forall\theta\in(-1, 1-\frac{\sqrt{2}}{2}).
\]
Therefore, we can set $\delta^{\prime}=1-\frac{\sqrt{2}}{2}$, and when $\theta
\in(-\delta^{\prime},\delta^{\prime})$, $\psi(\theta)<\theta^{2}$. According to \eqref{eq:par},
$a<\min(\frac{1}{2}, 4\delta^{\prime})=\frac{1}{2}$. We ran Algorithm 1 with
different values of $a$ and $\alpha$. Table \ref{tab:a} summarizes the running
time of the algorithm in different settings.

\begin{table}[h]
\caption{Running time of Algorithm 1 (in seconds)}%
\label{tab:a}%
\centering
\begin{tabular}
[c]{|l|l|l|l|l|}\hline
a & $\alpha=0.8$ & $\alpha=0.85$ & $\alpha=0.9$ & $\alpha=0.95$\\\hline
0.1 & 287.58 & 39.62 & 10.20 & 4.99\\
0.2 & 36.24 & 8.11 & 4.19 & 3.15\\
0.3 & 13.38 & 5.03 & 2.94 & 2.56\\
0.4 & 7.90 & 3.53 & 2.41 & 2.25\\
0.45 & 7.06 & 3.31 & 2.43 & 2.15\\\hline
\end{tabular}
\end{table}We observe that while $a$ is away from the upper bound $\frac{1}%
{2}$, the running time decreases as $a$ increases. We also observe that the
decreasing rate in $a$ is larger for smaller values of $\alpha$, which in
general implies greater curvature of the nonlinear boundary.

\subsubsection{Departure process of an M/G/$\infty$ queue}

We apply a variation of Algorithm 1 to simulate the steady state departure
process of an infinite server queue whose service times have infinite mean.
The details of the algorithm will be explained in Section \ref{sec:queue}. In
particular, we consider an infinite server queue having Poisson arrival
process with rate $1$, and Pareto service time distributions with probability
density function (pdf)
\[
f(v)=\beta v^{-(\beta+1)}I\{v\geq1\},
\]
for $\beta\in(1/2,1)$. Notice that we already know the departure process of
this $M/G/\infty$ queue should also be Poisson process with rate $1$,
therefore, this numerical experiment would help us verify the correctness of
our algorithm.

We sample the arrival process using Procedure A' and Procedure B' described in
Section \ref{sec:queue}, which are modifications of Procedure A and Procedure
B to adapt for the absolute value of the random walk. We truncate the length
of path at $10^{6}$ steps. We tried different pairs of parameters $\alpha$ and
$\beta$, and executed 1000 trials for each pair of $\alpha$ and $\beta$. We
count the number of departures between time 0 to 1 for each run and
construct the corresponding relative frequency bar plot. (Figure
\ref{fig:hist}). Figure \ref{fig:hist} suggests that the distribution of
simulated departures between time 0 and 1 indeed follows a Poisson distribution with rate
1. In particular, the distribution is independent of the values of $\alpha$
and $\beta$, which is consistent with what we expected. { We also conduct goodness of fit tests with the four groups of sampled data. The p-values for the tests are 0.2404, 0.2589, 0.4835, and 0.1137 respectively. Therefore the tests fail to reject that the generated samples are Poisson distributed.}

\begin{figure}[tbh]
\caption{Histograms comparison for sampled departure}%
\label{fig:hist}%
\centering
\begin{minipage}{.5\textwidth}
\centering
\includegraphics[width=1\linewidth]{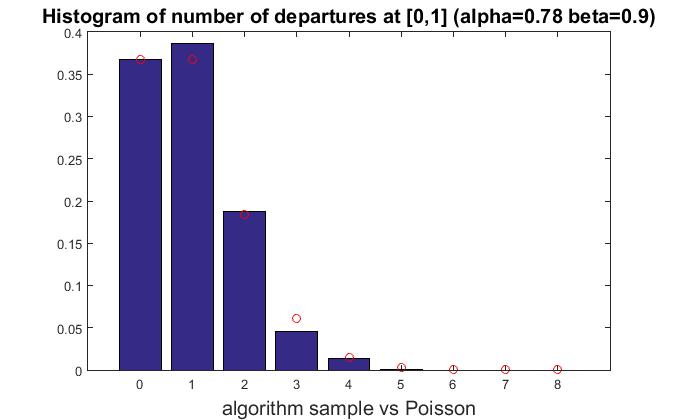}
\end{minipage}\begin{minipage}{.5\textwidth}
\centering
\includegraphics[width=1\linewidth]{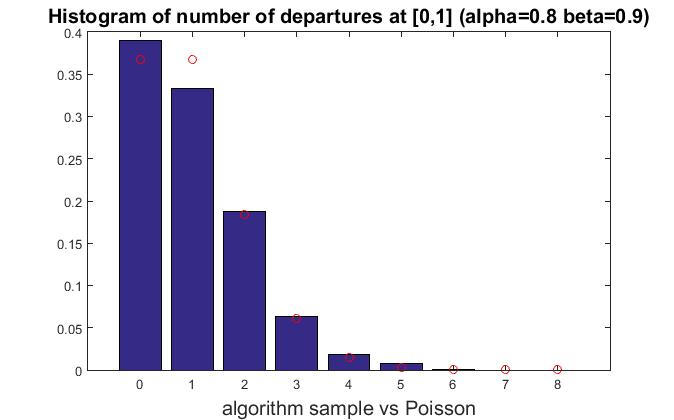}
\end{minipage}
\end{figure}

\begin{figure}[tbh]
\centering
\begin{minipage}{.5\textwidth}
\centering
\includegraphics[width=1\linewidth]{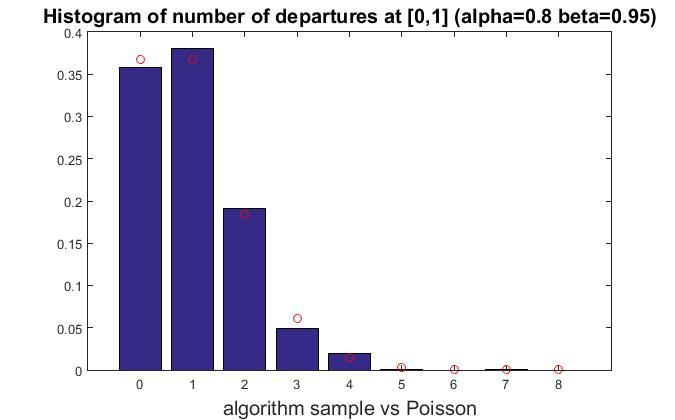}
\end{minipage}\begin{minipage}{.5\textwidth}
\centering
\includegraphics[width=1\linewidth]{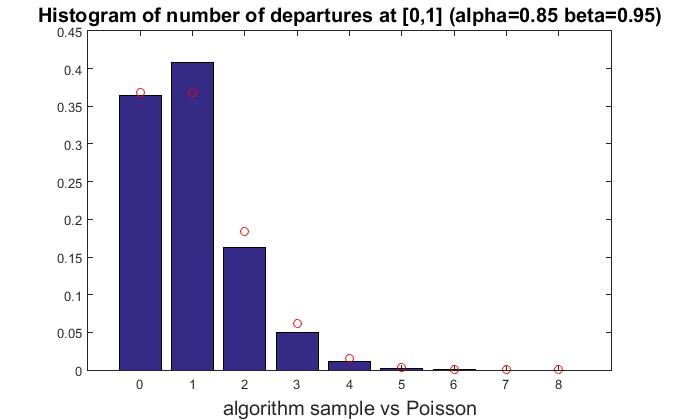}
\end{minipage}
\end{figure}

\section{Departure process of an infinite server queue}

\label{sec:queue} We finish the paper with an application of the algorithm
developed in Section \ref{sec:main} to sample the steady-state departure
process of an infinite server queue with general interarrival time and service
time distributions. We assume the interarrival times are i.i.d..
Independent of the arrival process, the service times are also i.i.d. and
may have infinite mean.

Suppose the system starts operating from the infinite past, then it would be
at stationarity at time $0$. We want to sample all the departures in the
interval $[0,h]$. We next introduce a point process representation of infinite
server queue to facilitate the explanation of simulation strategy. In
particular, we mark each arriving customer as a point in a 2-dimensional
space, where the $x$-coordinate records its arrival time and the
$y$-coordinate records its service time (service requirement). Based on this
point process representation, Figure \ref{fig:inf_main} provides a graphical
representation of the region that we are interested in simulation.
Specifically, to sample the departure process on $[0,h]$, we want to sample
all the points (customers) that fall into the shaded area. We further divide
the shaded area into two part, namely $H$ and $G$. Sampling the points that
fall into $G$ is easy. As $G$ is a bounded area, we can simply sample all the
arrivals between $0$ and $h$, and decide, using their service time
information, whether they fall into region $G$ or not. The challenge lies in
sampling the points in $H$, as it is an unbounded region.

\begin{figure}[tbh]
\caption{Point process representation of infinite server queue}%
\label{fig:inf_main}%
\centering
\includegraphics[width=0.5\textwidth]{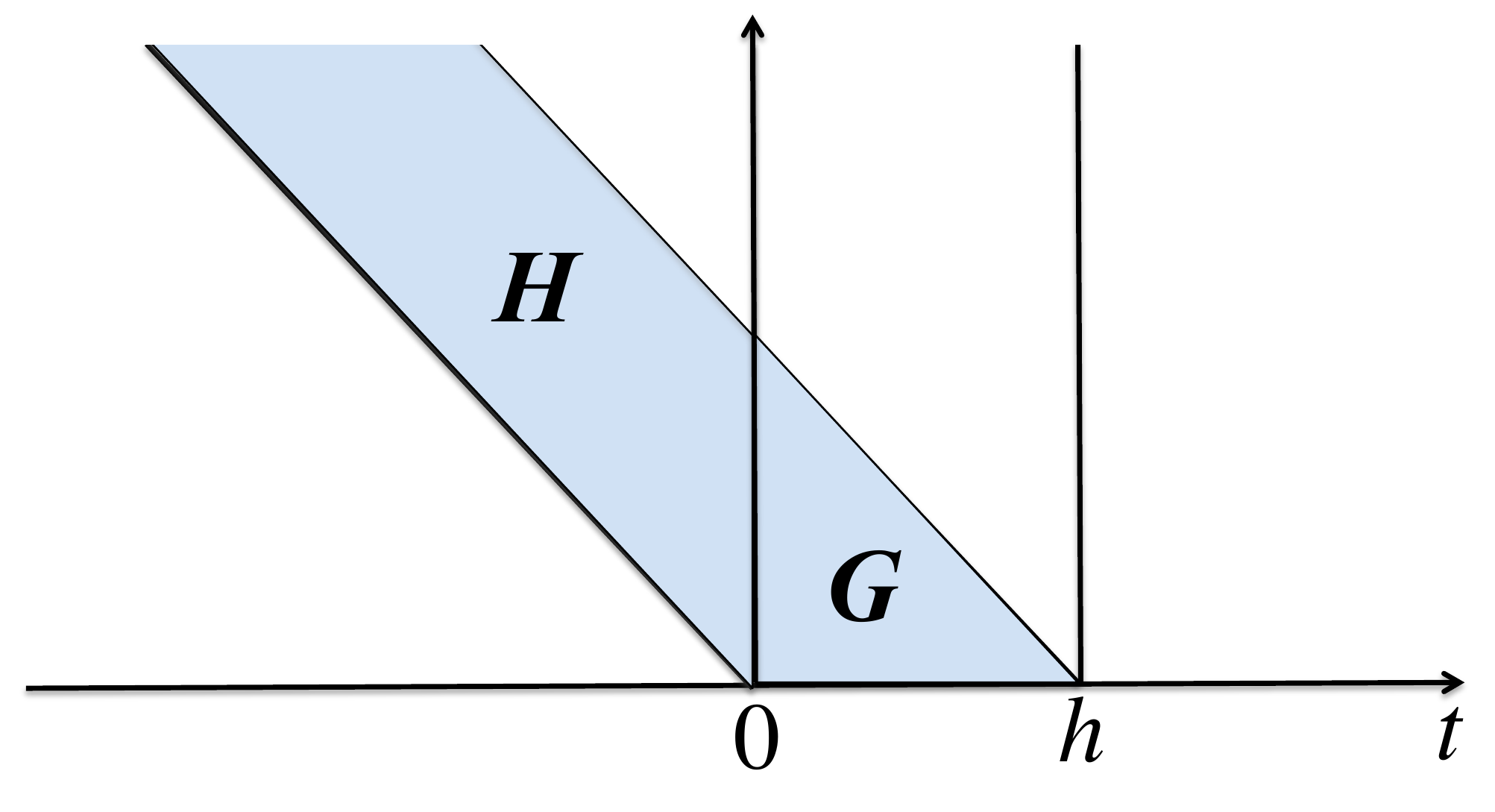}\end{figure}

We mark the points sequentially (according to their arrival times) backwards in
time from time $0$ as $(-A_{1}, V_{1})$, $(-A_{2}, V_{2})$, $\dots$, where
$-A_{n}$ is the arrival time of the $n$-th arrival counting backwards in time
and $V_{n} $ is his service time. Let $A_{0}:=0$. We then denote
$X_{n}:=A_{n-1}-A_{n}$, as the interarrival time between the $n$-th arrival
and the $(n-1)$-th arrival.

For simplicity of notation, we write
\[
\mathcal{H}:=\{(-A_{n}, V_{n}): A_{n}<V_{n}<A_{n}+h\}.
\]
It is the collection of points that fall into region $H$. If we can find a
random number $\Xi$ such that
\[
V_{n} < A_{n} \mbox{ or } V_{n} > A_{n}+h
\]
for $n \geq\Xi$, then we can sample the point process up to $\Xi$ and find
$\mathcal{H}$.

We further introduce an idea to separate the simulation of the arrival process
and the service time process. It requires us to find a sequence of
$\{\epsilon_{n}: n \geq1\}$, satisfying the following two properties:

\begin{enumerate}
\item There exists a well-defined random number $\Xi_{1}$, such that
\[
n\mu-\epsilon_{n}<A_{n}<n\mu+\epsilon_{n} \mbox{ for all } n \geq\Xi_{1}.%
\]

\item There exists a well-defined random number $\Xi_{2}$, such that
\[
V_{n} < n\mu-\epsilon_{n} \mbox{ or } V_{n} > n\mu+\epsilon_{n}+h
\mbox{ for all } n \geq\Xi_{2}.%
\]

\end{enumerate}

This allows us to find $\Xi_{1}$ and $\Xi_{2}$ separately and set $\Xi
=\max\{\Xi_{1},\Xi_{2}\}$. In particular, we can choose $\epsilon_{n}$
satisfying the following two conditions:

\begin{enumerate}
\item[C1)] $\sum_{n=1}^{\infty} P\left(  \left\vert A_{n}- n\mu\right\vert
>\epsilon_{n}\right)  <\infty$,

\item[C2)] $\sum_{n=1}^{\infty} P(V_{n}\in(n\mu-\epsilon_{n}, n\mu
+\epsilon_{n}+h))<\infty$.
\end{enumerate}

Then Borel-Cantelli Lemma guarantees that $\Xi_{1}$ and $\Xi_{2}$ are
``well-defined", i.e. finite almost surely.

We next introduce a specific choice of $\epsilon_{n}$ when the service
times follow a Pareto distribution with shape parameter $\beta\in
(1/2,1)$. We denote the pdf of $V_i$ as $f(\cdot)$, which takes the form
\begin{equation}
\label{eq:pareto}f(v)=\beta v^{-(\beta+1)}I\{v \geq1\}.
\end{equation}
We also write $\bar F(\cdot)$ as the tail cdf of $V_i$. We assume the
interarrival time has finite moment generating function in a neighborhood of
the origin. This is without loss of generality. Because if the interarrival
time is heavy-tailed, we can simulate a coupled infinite server queue with
truncated interarrival times, $X_{i}^{C}=\min\{X_{i}, C\}$, then it would
serve as an upper bound (in terms of the number of departures) of the original
infinite server queue in a path-by-path sense. Let $\mu:=E[X]$ denote the mean
interarrival time and $\sigma^{2}:=Var(X)$ denote its variance.

In this case, we can set $\epsilon_{n}=n^{\alpha}$ for $1/2<\alpha<\beta$. We
next show that our choice of $\epsilon_{n}$ satisfies C1) and C2)
respectively. We also explain how to find $\Xi_{1}$ and $\Xi_{2}$.

\subsection{Sampling of the arrival process and $\Xi_{1}$}

\begin{lemma}
\label{lm:bound1} If $\epsilon_{n}=n^{\alpha}$ for $\alpha>1/2$,
\[
\sum_{n=1}^{\infty} P\left(  \left\vert A_{n}- n\mu\right\vert >\epsilon
_{n}\right)  <\infty.
\]

\end{lemma}

\begin{proof}
We notice that $A_{n}=\sum_{i=1}^{n}X_{i}$ is a random walk with $X_{i}$ being
i.i.d. interarrival times with mean $\mu$, \emph{except the first one}.
$X_{1}$ follows the backward recurrent time distribution of the interarrival
time distribution. By moderate deviation principle \cite{Ganesh:2004}, we
have
\[
\frac{1}{n^{2\alpha-1}}P(\left\vert A_{n}-n\mu\right\vert >n^{\alpha
})\rightarrow-\frac{1}{2\sigma^{2}}.%
\]
As $2\alpha-1>0$, $\sum_{n=1}^{\infty}P\left(  \left\vert A_{n}-n\mu
\right\vert >n^{\alpha}\right)  <\infty$.
\end{proof}

Let $S_{n}=T_{n}-n\mu$. We notice that both $S_{n}$ and $-S_{n}$ are mean zero
random walks.
\[
P(|S_{n}|>n^{\alpha})\leq P(S_{n}>n^{\alpha})+P(-S_{n}>n^{\alpha}).
\]
Thus, we can apply a modified version of Algorithm 1 to find $\Xi_{1}$. In
particular, we define a modified sequence of record-breaking times as follows.
Let $T_{0}^{\prime}:=0$. For $k\geq1$, if $T_{k-1}^{\prime}<\infty$,
\begin{align*}
T_{k}^{\prime}  &  :=\inf\left\{  n>T_{k-1}^{\prime}: S_{n}>S_{T_{k-1}%
^{\prime}}+a(n-T_{k-1}^{\prime})^{\alpha}+b(n- T_{k-1}^{\prime})^{1-\alpha
}\right. \\
&  \left.  \mbox{ or } S_{n}<S_{T_{k-1}^{\prime}}-a(n-T_{k-1}^{\prime
})^{\alpha}-b(n- T_{k-1}^{\prime})^{1-\alpha}\right\};
\end{align*}
else, $T_{k}^{\prime} =\infty$. { Then the modified version of Algorithm 1 goes as follows.\\

\noindent\textbf{Algorithm $1^{\prime}$.} Sampling $\Xi$ together with
$(X_{i}: 1\leq i\leq \Xi)$.

\begin{enumerate}
\item[i)] Initialize $T_{0}=0$, $k=1$.

\item[ii)] For $T_{k-1}<\infty$, sample $J \sim$Bernoulli$(P(T_{k}=\infty|T_{k-1}))$.
If $J=0$, sample $(X_{i}: i=T_{k-1}+1, \dots, T_{k})$ conditional
on $T_{k}<\infty$. Set $k=k+1$ and go back to step ii); otherwise ($J=1$), set
$\Xi_1=T_{k-1}$ and go to step iii). (see Procedure $A^{\prime}$)

\item[iii)] Apply Procedure $C$ (detailed in Section 4.2) to sample $\Xi_2$.

\item[iv)] Set $\Xi=\max\{\Xi_1, \Xi_2\}$.
If $\Xi>\Xi_1$, sample $(X_{i}: i=T_{k-1}+1, \dots, \Xi )$ 
conditional on $T_{k}=\infty$. (see Procedure $B^{\prime}$)
\end{enumerate}
}

We also modify Procedure A and Procedure B as
follows.\newline

\noindent\textbf{Procedure $A^{\prime}$.} Sampling $J^{\prime}$ with
$P(J^{\prime}=1)=P(T_{1}^{\prime}=\infty)$, if $J^{\prime}=0$, output $(X_{1},
\dots, X_{T_{1}^{\prime}})$.

\begin{enumerate}
\item[i)] Sample a random time $N$ with pmf \eqref{eq1}. Let $\theta
_{N}=a2^{N(\alpha-1)-2}$. Sample $U_{1}\sim$Uniform$[0,1]$. If $U_{1}\leq1/2$,
go to step ii a), else go to step ii b).

\item[ii a)] Generate $X_{1},\dots, X_{2^{N+1}-1}$ under exponential tilting
with tilting parameter $\theta_{N}$. Let
\[
T_{1}^{\prime} =\inf\{n\geq1: \left\vert S_{n} \right\vert > an^{\alpha
}+bn^{1-\alpha} \} \wedge2^{N}.%
\]
.

\item[ii b)] Generate $X_{1},\dots, X_{2^{N+1}-1}$ under exponential tilting
with tilting parameter $-\theta_{N}$. Let
\[
T_{1}^{\prime}=\inf\{n\geq1: \left\vert S_{n} \right\vert > an^{\alpha
}+bn^{1-\alpha} \} \wedge2^{N}.
\]

\item[iii)] Generate $U_{2}\sim$Uniform$[0,1]$. If
\[
U_{2}\leq\frac{\left(  \frac{1}{2}\exp(\theta_{N} S_{T_{1}^{\prime}}%
-\psi(\theta_{N}) T_{1}^{\prime}) +\frac{1}{2}\exp(-\theta_{N} S_{T_{1}%
^{\prime}}-\psi(-\theta_{N}) T_{1}^{\prime})\right)  ^{-1}}{p(N) }\times
I\left\{  T_{1}^{\prime}\in[2^{N}, 2^{N+1})\right\},
\]
then set $J^{\prime}=0$ and output $(X_{1}, X_{2}, \dots, X_{T_{1}^{\prime}}%
)$; else, set $J^{\prime}=1$.
\end{enumerate}

\bigskip

\begin{proposition}
\label{prop:step1} In Procedure $A^{\prime}$, $J^{\prime}$ is a Bernoulli
random variable with probability of success $P(T_{1}^{\prime}=\infty)$. If
$J=0 $, the output $(X_{1}, X_{2}, \dots, X_{T_{1}^{\prime}})$ follows the
distribution of $(X_{1}, X_{2}, \dots, X_{T_{1}^{\prime}})$ conditional on
$T_{1}^{\prime}<\infty$.
\end{proposition}

The proof of Proposition \ref{th:step1} follows exactly the same line of
analysis as the proof of Theorem \ref{th:step1}. We shall omit it
here.\newline


Let
\[
L^{\prime}(n)=\inf\left\{  k>n: S_{k} \in\left(  -\frac{a}{4}k^{\alpha},
\frac{a}{4} k^{\alpha}\right)  \mbox{ or } S_{k}>ak^{\alpha}+bk^{1-\alpha}
\mbox{ or } S_{k}<-ak^{\alpha}-b k^{1-\alpha}\right\}  .
\]

\bigskip

\noindent\textbf{Procedure $B^{\prime}$.} Sampling $(X_{1}, \dots, X_{n})$
conditional on $T_{1}^{\prime}=\infty$.

\begin{enumerate}
\item[i)] Sample $(X_{1}, \dots, X_{n})$ under the nominal distribution
$P(\cdot) $.

\item[ii)] If $\max_{1\leq k\leq n}\{S_{k}-ak^{\alpha}-bk^{1-\alpha}\}>0$ or
$\min_{1\leq k\leq n}\{S_{k}+ak^{\alpha}+bk^{1-\alpha}\}<0$, go back to step
i); else, go to step iii).

\item[iii)] Sample $L^{\prime}(n)$ and $(X_{n+1}, \dots, X_{L^{\prime}(n)})$
under the nominal distribution $P(\cdot)$. If $|S_{L^{\prime}(n)}|>aL^{\prime
}(n)^{\alpha}+bL^{\prime}(n)^{1-\alpha}$, go back to step i); else, go to step iv).

\item[iv)] Sample $\tilde N$ with probability mass function $p_{L(n)}$ defined
in \eqref{eq2}. Set $\tilde\theta_{\tilde N}=2^{\tilde N-2} a\left(  2^{\tilde
N}+L(n)\right)  ^{\alpha}$. Sample $U_{1}\sim$Uniform$[0,1]$. If $U_{1}<1/2$,
go to step v a); else, go to step v b).

\item[v a)] Generate $\tilde X_{1}, \tilde X_{2}, \dots, \tilde X_{2^{\tilde
N+1}-1}$ under exponential tilting with tilting parameter $\tilde
\theta_{\tilde N}$. Let
\[
\tilde T^{\prime}=\inf\left\{  n\geq1: \left\vert S_{L^{\prime}(n)}+\tilde
S_{k} \right\vert >a(k+L^{\prime}(n))^{\alpha}+b(k+L^{\prime}(n))^{1-\alpha
}\right\}  \wedge2^{\tilde N+1}.
\]

\item[v b)] Generate $\tilde X_{1}, \tilde X_{2}, \dots, \tilde X_{2^{\tilde
N+1}-1}$ under exponential tilting with tilting parameter $-\tilde
\theta_{\tilde N}$. Let
\[
\tilde T^{\prime}=\inf\left\{  n\geq1: \left\vert S_{L^{\prime}(n)}+\tilde
S_{k} \right\vert >a(k+L^{\prime}(n))^{\alpha}+b(k+L^{\prime}(n))^{1-\alpha
}\right\}  \wedge2^{\tilde N+1}.
\]

\item[vi)] Sample $U_{2}\sim$Uniform$[0,1]$. If
\[
U_{2}\leq\frac{\left(  \frac{1}{2} \exp\left(  \tilde\theta_{\tilde
N}\tilde S_{\tilde T^{\prime}}-\tilde\psi(\tilde\theta_{\tilde N}) \right)  +\frac
{1}{2}\exp\left(  -\tilde\theta_{\tilde N}\tilde S_{\tilde T^{\prime}} - \tilde
\psi(-\tilde\theta_{\tilde N})\right)  \right)  ^{-1}}{p_{t}(\tilde N)}\times
I\left\{  \tilde T^{\prime}\in\left[  2^{\tilde N}, 2^{\tilde N+1}\right)
\right\}, 
\]
set $\tilde J^{\prime}=0$ and go back to Step i); else, set $\tilde J^{\prime
}=1$ and output $(X_{1}, \dots, X_{n})$.\newline
\end{enumerate}

\bigskip

\begin{proposition}
\label{prop:step2} The output of Procedure $B^{\prime}$ follows the
distribution of $(X_{1}, \dots, X_{n})$ conditional on $T_{1}^{\prime}=\infty$.
\end{proposition}

The proof of Proposition \ref{prop:step2} follows exactly the same line of
analysis as the proof of Theorem \ref{th:step2}. We shall omit it
here.\newline

\subsection{Sampling of the service time process and $\Xi_{2}$}

\begin{lemma}
\label{lm:bound2} If $\epsilon_{n}=n^{\alpha}$ for $1/2<\alpha<\beta$,
\[
\sum_{n=1}^{\infty} P(V_{n}\in(n\mu-\epsilon_{n}, n\mu+\epsilon_{n}%
+h))<\infty.
\]

\end{lemma}

\begin{proof}%
\begin{align*}
P(V_{n}\in(n\mu-\epsilon_{n}, n\mu+\epsilon_{n}+h))  &  =\bar F(n\mu
-\epsilon_{n})-\bar F(n\mu+\epsilon_{n}+h)\\
&  \leq\frac{\beta}{(n\mu- n^{\alpha})^{(\beta+1)}} (2 n^{\alpha}+h)\\
&  =\frac{\beta(2+hn^{-\alpha})}{n^{\beta+1-\alpha}(\mu-n^{-(\beta-\alpha
)})^{\beta+1}}.
\end{align*}
As $\beta+1-\alpha>1$,
\[
\sum_{n=1}^{\infty} \frac{\beta(2+hn^{-\alpha})}{n^{\beta+1-\alpha}%
(\mu-n^{\alpha-\beta})^{\beta+1}}<\infty.
\]

\end{proof}

To find $\Xi_{2}$, we use a similar record-breaker idea. In particular, we say
$V_{n}$ is a record-breaker if
\[
V_{n} \in(n\mu- \epsilon, n\mu+\epsilon_{n}+h).
\]
The idea is to find the record-breakers sequentially until there
are no more record-breakers. Specifically, let $K_{0}:=0$ and if
$K_{i-1}<\infty$,
\[
K_{i}=\inf\{n>K_{i-1}: V_{n} \in(n\mu- \epsilon, n\mu+\epsilon_{n}+h)\}.
\]
if $K_{i-1}=\infty$, $K_{i}=\infty$. Let $\tau=\min\{ i>0: K_{i}=\infty\}$.
Then we can set $\Xi_{2}=K_{\tau-1}$.

The task now is to find $K_{i}$'s one by one. We start with $K_{1}$.
\[
P(K_{1}=\infty)=\prod_{n=1}^{\infty}\left(  1-P(V_{n}\in(n\mu-\epsilon_{n},
n\mu+\epsilon_{n}+h))\right).
\]

Let
\[
u(k)=\prod_{n=1}^{k}\left(  1-P(V_{n}\in(n\mu-\epsilon_{n}, n\mu+\epsilon
_{n}+h))\right).
\]
Then we have $P(K_{1}=\infty)<u(k+1)<u(k)$ for any $k\geq1$ and $\lim
_{k\rightarrow\infty}u(k)=P(K_{1}=\infty)$. We also notice that
$u(k)-u(k-1)=P(K_{1}=k)$.

From the proof of Lemma \ref{lm:bound2}, we have for $n>(2/\mu)^{1/(\beta
-\alpha)}$,
\[
P(V_{n}\in(n\mu-\epsilon_{n}, n\mu+\epsilon_{n}+h))<\frac{2(2+h)\beta}{\mu
}\frac{1}{n^{\beta+1-\alpha}}.%
\]
Then for $k^{*}$ large enough such that $k^{*}>(2/\mu)^{1/(\beta-\alpha)}$,
and $\frac{2(2+h)\beta}{\mu}\frac{1}{n^{\beta+1-\alpha}}<1$, we have for
$k>k^{*}$.
\begin{align*}
&  \prod_{n=k+1}^{\infty}\left(  1-P(V_{n}\in(n\mu-\epsilon_{n}, n\mu
+\epsilon_{n}+h))\right) \\
&  \geq\prod_{n=k+1}^{\infty}\left(  1-\frac{2(2+h)\beta}{\mu}\frac
{1}{n^{\beta+1-\alpha}}\right) \\
&  \geq\exp\left(  -\frac{(2+h)\beta}{\mu}\sum_{n=k+1}^{\infty}\frac
{1}{n^{\beta+1-\alpha}}\right) \\
&  \geq\exp\left(  -\frac{(2+h)\beta}{\mu} (k+1)^{-(\beta-\alpha)}\right).
\end{align*}
Let $l(k)=0$ for $k<k^{*}$ and
\[
l(k)=u(k)\exp\left(  -\frac{2(2+h)\beta}{\mu} (k+1)^{-(\beta-\alpha)}\right)
\]
for $k>k^{*}$. Then we have $l(k)\leq l(k+1)<P(K_{1}=\infty)$ and
$\lim_{k\rightarrow\infty}l(k)=P(K_{1}=\infty)$.

Similarly, given $K_{i-1}=m<\infty$, we can construct the sequences of upper
and lower bounds for $P(K_{i}=\infty|K_{i-1}=m)$ as
\[
u_{m}(k)=\prod_{n=m+1}^{k}\left(  1-P(V_{n}\in(n\mu-\epsilon_{n},
n\mu+\epsilon_{n}+h))\right)
\]
for $k > m$, and
\[
l_{m}(k)=u_{m}(k)\exp\left(  -\frac{(2+h)\beta}{\mu} (k+1)^{-(\beta-\alpha
)}\right)  .
\]

Based on the sequence of lower and upper bounds, given $K_{i-1}=m$, we can
sample $K_{i}$ using the following iterative procedure.\newline

\noindent\textbf{Procedure C.} Sample $K_{i}$ conditional on $K_{i-1}=m$.

\begin{enumerate}
\item[i)] Generate $U\sim$Uniform$[0,1]$. Set $k=m+1$. Calculate $u_{m}(k)$
and $l_{m}(k)$.

\item[ii)] While $l_{m}(k)<U<u_{m}(k)$\newline Set $k=k+1$. Update $u_{m}(k)$
and $l_{m}(k)$.\newline end While.

\item[iii)] If $U<l_{m}(k)$, output $K_{i}=\infty$; else, output $K_{i}=k$.
\end{enumerate}

Once we find the values of $K_{i}$'s, sampling $V_{n}$'s conditional on the
information of $K_{i}$'s is straightforward.

We next provide some comments about the running time of Procedure C.
Let $\Phi_i$ denote the number of iterations in Procedure C to generate $K_i$.
We shall show that while $P(\Phi_i<\infty)=1$, while $E[\Phi_i] =\infty$.
Take $\Phi_1$ as an example:
\begin{align*}
P( \Phi_1> n)&=P(K_1 > n)\\
&=P(l_1(n) <U<u_1(n))\\
&\ge u_1(n) \left(1- \exp\left(-\frac{2(2+h)\beta}{\mu} (n+1)^{-(\beta-\alpha)}\right)\right),\\
\end{align*}
with 
$$1- \exp\left(-\frac{2(2+h)\beta}{\mu} (n+1)^{-(\beta-\alpha)}\right)=O(n^{-(\beta-\alpha)}),$$ 
and $u_1(n) \ge P(K_1=\infty)$ for any $n\ge 1$. As $1<\beta-\alpha<1$, we have $P(K_1<\infty)=1$, but $\sum_{n=1}^\infty P(K_1 > n)=\infty$. Thus, $P(\Phi_1<\infty)=1$, but $E[\Phi_1]=\infty$.\\

The fact that the Procedure C has infinite expected termination time may be unavoidable in the following sense. In the absence of additional assumptions on the traffic feeding into the infinite server queue, any algorithm which simulates stationary departures during, say, time interval $[0,1]$, must be able to directly simulate the earliest arrival, from the infinite past, which departs in $[0,1]$. If such arrival is simulated sequentially backwards in time, we now argue that the expected time to detect such arrival must be infinite. Assuming, for simplicity, deterministic inter-arrival times equal to 1, and letting $-T<0$ be the time at which such earliest arrival occurs, then we have that
\begin{align*}
P(T > n) &\ge P(\cup_{k=n+1}^\infty \{V_k \in [k,k+1]\})\\
&\ge (1-P(V>n)) \sum_{k=n+1}^\infty P(V_k \in [k,k+1])\\
&= (1-P(V>n))P(V > n+1).
\end{align*}
As $\sum_{n=0}^\infty P(V>n) =\infty$, we must have that $E[T]=\infty$.

\begin{remark}
Based on our analysis above, in general, there is a trade-off between $\Xi
_{1}$ and $\Xi_{2}$ in terms of $\epsilon_{n}$. The smaller $\epsilon_{n}$ is,
the larger the value of $\Xi_{1}$ and the smaller the value of $\Xi_{2}$.
\end{remark}

\bibliographystyle{plain}
\bibliography{exact_ref}

\end{document}